\documentclass[12pt]{amsart}
\usepackage{amsmath,amsthm,amsfonts,amscd,amssymb,eucal,latexsym,mathrsfs,enumerate}

\usepackage{hyperref}

\newcommand{\ovl}{\overline}

\newcommand{\vp}{\varepsilon}

\newcommand{\Hs}{\mathcal H}
\newcommand{\Ks}{\mathcal K}

\numberwithin{equation}{section}

\theoremstyle{plain}
\newtheorem{lem}{Lemma}[section]
\newtheorem{thm}[lem]{Theorem}
\newtheorem{cor}[lem]{Corollary}

\theoremstyle{definition}

\newtheorem{definition}[lem]{Definition}

\theoremstyle{remark}

\title{C$^*$-algebras nearly contained in type $\mathrm{I}$ algebras}

\author[E.~Christensen]{Erik Christensen}
\address{\hskip-\parindent
Erik Christensen, Institute for Mathematiske Fag, University of Copenhagen, Copenhagen, Demark.}
\email{echris@math.ku.dk}
\author[A.~M.~Sinclair]{Allan M.~Sinclair}
\address{\hskip-\parindent
Allan M.~Sinclair, School of Mathematics, University of Edinburgh, JCMB, King's Buildings, Mayfield Road, Edinburgh, EH9 3JZ, Scotland.}
\email{a.sinclair@ed.ac.uk}
\author[R.~R.~Smith]{Roger R.~Smith}
\address{\hskip-\parindent
Roger R.~Smith, Department of Mathematics, Texas A{\&}M University,
College Station TX 77843-3368, U.S.A.}
\email{rsmith@math.tamu.edu}
\author[S.~White]{Stuart White}
\address{\hskip-\parindent
Stuart White, School of Mathematics and Statistics, University of Glasgow, 
University Gardens, Glasgow Q12 8QW, Scotland.}
\email{stuart.white@glasgow.ac.uk}
\date{\today}

\begin{document}

\maketitle

\begin{abstract}
In this paper we consider near inclusions $A\subseteq_\gamma B$ of C$^*$-algebras. We show that if $B$ is a separable type $\mathrm{I}$ C$^*$-algebra and $A$ satisfies Kadison's similarity problem, then $A$ is also type $\mathrm{I}$ and use this to obtain an embedding of $A$ into $B$.
\end{abstract}

\section{Introduction}
This paper is concerned with obtaining embeddings from near inclusions of C$^*$-algebras.  Given two C$^*$-algebras $A$ and $B$ concretely represented on the same Hilbert space, we say that $A$ is nearly contained in $B$ if every element of the unit ball of $A$ can be well approximated by an operator from $B$ (see Definition \ref{DEFNI} below for the precise definition).  One natural way to produce a near containment of C$^*$-algebras is to take a genuine inclusion $A_0\subseteq B$ and then set $A=uA_0u^*$ for some unitary $u\in\mathbb B(\mathcal{H})$ which is close to $I_{\mathcal{H}}$.  In this case, $A$ certainly embeds into $B$ and in general it is an open question whether, given any sufficiently small near inclusion of C$^*$-algebras, there must be an embedding of the `smaller' algebra into the nearly containing algebra.  In this paper we investigate the situation when the `larger' algebra is separable and type $\mathrm{I}$. Inspired by earlier work of Phillips \cite{Ph}, we provide conditions which imply that the `smaller' algebra is also type $\mathrm{I}$. From this we obtain an embedding of the `smaller' algebra into the `larger algebra'.

Near containments of operator algebras were introduced by the first named author in his work \cite{Ch2} on perturbations of operator algebras. In \cite{KK}, Kadison and Kastler introduced a metric on the collection of all operator subalgebras of $\mathbb B(\mathcal{H})$, conjecturing that sufficiently close C$^*$-algebras should be spatially isomorphic. Qualitatively, two algebras are close in this metric if and only if they are nearly contained in each other.  Combining Raeburn and Taylor's work \cite{RT} (or alternatively Johnson's work \cite{Johnson.P}) with \cite{Ch-P1} gives a complete answer to Kadison and Kastler's conjecture when one algebra is an injective von Neumann algebra.  Another approach to this result was obtained in \cite{Ch2} via near inclusions and this shows that any algebra close to an injective von Neumann algebra is spatially isomorphic to it via a unitary which is close to the identity. Recently the authors and Winter have shown that sufficiently close C$^*$-algebras on a separable Hilbert space are spatially isomorphic when one algebra is separable and nuclear, \cite{CSSWW} (see also \cite{CSSWWPNAS}). This provides a C$^*$-algebraic analogue of the original result for injective von Neumann algebras. Separability is a necessary hypothesis in the C$^*$-algebraic
situation due to the examples of \cite{Christensen.CounterExamples}. Another key difference which arises in the C$^*$-algebraic setting is that one cannot expect to produce a unitary close to the identity implementing a spatial isomorphism between close C$^*$-algebras; counterexamples are given in \cite{Johnson.PerturbationExample}.

We note that for pairs of separable C$^*$-algebras or of von Neumann
algebras, no examples are known of close algebras that are not isomorphic,
or of near containments $A\subseteq_{\gamma} B$ that do not admit an
embedding of $A$ into $B$. Positive results have been obtained by Phillips
and Raeburn \cite{Phillips.PerturbationAF,Phillips.Perturbations2} for the cases of AF C$^*$-algebras and algebras with
continuous trace, and more generally for all separable nuclear
C$^*$-algebras in \cite{CSSWW}.

For von Neumann subalgebras of $\mathbb B(\mathcal{H})$, \cite{Ch2} gives two general situations in which near containments give rise to embeddings. We choose a formulation in (1) which is different from the original one but equivalent to it.

\begin{enumerate}[(1)]
\item\label{I1} Suppose that $M$ is a von Neumann algebra satisfying Kadison's similarity property (for example, a properly infinite von Neumann algebra, or a II$_1$ factor with Property $\Gamma$) and $M$ is nearly contained in an injective von Neumann algebra $N$. Then there is a spatial embedding of $M$ into $N$.  This embedding can be implemented by a unitary close to the identity, where the estimates depend on the size of the original near inclusion and constants that arise in Kadison's similarity property.
\item\label{I2} Suppose that an injective von Neumann algebra $M$ is nearly contained in a von Neumann algebra $N$. Then there is a spatial embedding of $M$ into $N$, and a unitary implementing this embedding can be chosen close to the identity operator (in terms of the size of the original near inclusion).
\end{enumerate}
\noindent These statements are equivalent. Indeed \cite{Ch2} deduces the second from the first, using a commutation argument involving the double commutant theorem and a very similar argument can be used to show that the first statement implies the second statement. Given the complete analogue for separable nuclear C$^*$-algebras of the results for close von Neumann algebras, it is natural to also ask for C$^*$-algebraic versions of (\ref{I1}) and (\ref{I2}). For item (\ref{I2}) this is achieved in \cite[Section 2]{HKW}; in this paper we make progress towards a C$^*$-algebraic version of (\ref{I1}). 

As the double commutant theorem does not apply, our strategy is to find
conditions under which we can apply the results from \cite{CSSWW} for near
inclusions of pairs of nuclear C$^*$-algebras, i.e. a separable nuclear
C$^*$-algebra which is nearly contained in a nuclear C$^*$-algebra $B$ must
embed into $B$.  Thus we look for additional hypotheses on a nuclear
C$^*$-algebra $B$ that imply that every algebra with the similarity property
which is nearly contained in $B$ is automatically nuclear.  The class of type
$\mathrm{I}$ C$^*$-algebras is the largest class of C$^*$-algebras for which
every C$^*$-subalgebra is nuclear and so this is the least restrictive
hypothesis on $B$ for which this method could apply.   In Theorem \ref{thm4.5}
we show that a C$^*$-algebra which has the similarity property and is nearly
contained in a separable type $\mathrm{I}$ C$^*$-algebra is also type
$\mathrm{I}$.  We do this by showing that algebras nearly contained in a
separable liminary C$^*$-algebra are also liminary (Theorem \ref{thm3.1}), and
then using Phillips's methods to transfer a decomposition series for a type
$\mathrm{I}$ C$^*$-algebra $B$ to algebras nearly contained in $B$.  At this
stage we need to assume the `smaller algebra' has the similiarity property in
order to produce ideals in the `smaller' algebra corresponding to those in $B$
and this is a new feature compared with Phillips's original result that
C$^*$-algebras close to separable type $\mathrm{I}$ algebras are also type
$\mathrm{I}$.

\subsection*{Acknowledgments}

Our study of type I algebras was initiated after a visit by SW to Leeds
University. He would like to thank Garth Dales, Matt Daws and Charles Read for
stimulating conversations.  This paper was completed while SW was a visiting
researcher at the CRM in Barcelona.  He would like to thank the CRM for its
hospitality.

\section{Near Inclusions and property $D_k$}\label{sec2}

In this section we recall the definitions of a near inclusion and give  formulation of Kadison's similarity property suitable for use in Section \ref{sec4}.
\begin{definition}\label{DEFNI}
Let $A$ and $B$ be C$^*$-subalgebras of $\mathbb B(\Hs)$. For $\gamma>0$, write
$A\subseteq_\gamma B$ if, for each $x\in A$, there exists $y\in B$ with
$\|x-y\|\leq\gamma\|x\|$. Write $A\subset_\gamma B$ if there exists
$\gamma'<\gamma$ with $A\subseteq_{\gamma'}B$.  At one point in the paper we
need near inclusions when $A$ and $B$ are concrete operator spaces; these are
defined in exactly the same way.
\end{definition}

In general it is unknown whether a near inclusion $A \subseteq_\gamma B$ for two C$^*$-algebras on a Hilbert space $\Hs$ induces a near inclusion $B' \subseteq_{\gamma'} A'$ of their commutants.
If $A\subseteq {\mathbb{B}}(\Hs)$ is a C$^*$-algebra then each $x\in {\mathbb{B}}(\Hs)$ induces a
spatial derivation ${\mathrm{ad}}(x):A \to {\mathbb{B}}(\Hs)$, where
${\mathrm{ad}}(x)(a)=[a,x]=ax-xa$ for $a\in A$, and a simple calculation shows
that
$\|{\mathrm{ad}}(x)|_A\|\leq 2d(x,A')$. If, for some $k>0$, the reverse inequality
\begin{equation}\label{eq2.10}
d(x,A')\leq k\|{\mathrm{ad}}(x)|_A\|
\end{equation}
holds for all $x\in {\mathbb{B}}(\Hs)$, then $A$ is said to have the {\em{local
distance property}} $LD_k$ on $\Hs$.  This property was used in \cite{Ch2} to define the distance property $D_k$ for a C$^*$-algebra in terms of representations. However, there was an implicit assumption of nondegeneracy of representations, and so we now refine this definition. This is important for our work here as potentially degenerate representations will arise subsequently.
\begin{definition}\label{defn2.2}
A C$^*$-algebra $A$ has the distance property $D_k$ for some $k>0$ if, for every nondegenerate representation $\pi : A \to \mathbb{B}(\Hs)$, $\pi(A)$ has property $LD_k$ on $\Hs$. If $A$ is also a von Neumann algebra, then restriction to normal degenerate representations will define the normal version of property $D_k$, denoted $D_k^*$.
\end{definition}
\noindent One consequence of property $LD_k$, which we will use below, is that a near inclusion $A \subseteq_\gamma B$ induces a near inclusion $B' \subseteq_{2k\gamma} A'$ when $A$ has property $LD_k$, see \cite[Proposition 2.5]{CSSW}.  

Considerations of distance properties first arose in Arveson's study of nest algebras \cite{Ar} where he showed that
\begin{equation}\label{eq4.2}
 d(x,A') = \frac12\|{\mathrm{ad}}(x)|_A\|_{{\mathrm{cb}}},\qquad x\in {\mathbb{B}}(\Hs),
\end{equation}
for nest algebras $A$ (see also \cite{Ch}). The significance of this concept is that all derivations $\delta\colon \ A\to {\mathbb{B}}(\Hs)$ are inner if and only if $A$ has property $LD_k$ on $\Hs$ for some $k>0$. In this way, the distance property characterizes when all bounded derivations of a C$^*$-algebra $A$ into the bounded operators on some Hilbert space are inner. In \cite{Ki}, Kirchberg shows that this property is in turn equivalent to Kadison's similarity property for $A$, i.e. every bounded representation $A\rightarrow\mathbb B(\Hs)$ is similar to a $^*$-representation.  In summary, $A$ has Kadison's similarity property if and only if there is some $k>0$ such that $A$ has the distance property $D_k$.

The following three lemmas handle the adjustments which are necessary to consider algebras in general position. The disparity between the numbers $k$ and $k+1$ below is accounted for by the issue of degenerate versus nondegenerate representations. All ideals in the paper are assumed to be norm closed.

\begin{lem}\label{lem4.1}
Let $A\subseteq {\mathbb{B}}(\Hs)$ be a ${\mathrm{C}}^*$-algebra with an ideal $J$.
\begin{itemize}
\item[\rm (i)] If $A$ has property $LD_k$ then $J$ has property $LD_{k+1}$.
\item[\rm (ii)] If $A$ has property $D_k$ then $J$ has property $D_{k}$.
\end{itemize}
\end{lem}

\begin{proof}
(i) We first consider the special case when $A$ is a von~Neumann algebra and $J$ is a weakly closed ideal. Then there is a central projection $z\in A$ so that $J=Az$. It will be convenient to write operators in ${\mathbb{B}}(\Hs)$ as $2\times 2$ matrices relative to the decomposition $I_{\mathcal H}=z+z^\bot$. Then operators in $A$ have the form $\left(\begin{smallmatrix} j&0\\ 0&*\end{smallmatrix}\right)$ for $j\in J$. Consider an element $t = \left(\begin{smallmatrix} t_1&t_2\\ t_3&t_4\end{smallmatrix}\right)\in {\mathbb{B}}(\Hs)$ such that ${\mathrm{ad}}(t)|_J$ has norm 1. We must find an operator $j'\in J'$ such that $\|t-j'\| \le k+1$. Since $\left(\begin{smallmatrix} 0&0\\ 0&t_4\end{smallmatrix}\right) \in J'$ we can assume that $t_4=0$. For any $j\in J$ of norm 1, its commutator with $t$ is $\left(\begin{smallmatrix} [j,t_1]&jt_2\\ -t_3j&0\end{smallmatrix}\right)$ and so $\left(\begin{smallmatrix} t_1&0\\ 0&0\end{smallmatrix}\right)$ induces a derivation on $J$ of norm at most 1. This agrees with the induced derivation on $A$ and so it is within a distance $k$ of some operator $j'\in A'\subseteq J'$. Taking $j=z=\left(\begin{smallmatrix} 1&0\\ 0&0\end{smallmatrix}\right)$ above, we see that $\left(\begin{smallmatrix} 0&t_2\\ -t_3&0\end{smallmatrix}\right)$ has norm at most 1, and the same is then true for $\left(\begin{smallmatrix} 0&t_2\\ t_3&0\end{smallmatrix}\right)$. It follows that $\|t-j'\|\le k+1$.

For a general inclusion $J\subseteq A$, property $LD_k$ passes to the weak closure, and the result follows from the special case applied to the weak closures of $J\subseteq A$.

\noindent (ii) If $\pi\colon \ J\to {\mathbb{B}}(\Hs)$ is a nondegenerate representation, then it extends to a nondegenerate representation, also denoted $\pi$, of $A$ on the same Hilbert space $\Hs$, \cite[Prop. 2.10.4]{Dix}. Since $A$ has property $D_k$, $\pi(A)$ has property $LD_k$ on $\Hs$, so $\pi(J)$ has property $LD_{k}$ on $\Hs$ because the weak closures of $\pi(J)$ and of $\pi(A)$ coincide in this situation. Since $\pi$ was arbitrary, we see that $J$ has property $D_{k}$.\end{proof} 

\begin{lem}\label{lem2.3b}
Let $A$ be a ${\mathrm{C}}^*$-algebra with property $D_k$. Then, for any (possibly degenerate) representation $\pi :A\to \mathbb{B}(\Hs)$, $\pi(A)$ has property $LD_{k+1}$ on $\Hs$.
\end{lem}

\begin{proof}
Consider a representation $\pi :A\to \mathbb{B}(\Hs)$ and let $p$ be the support projection of $\pi(A)$, which is the unit for the weak closure $M$. The restriction of $\pi$ to $p\Hs$ is nondegenerate so $\pi(A)$ has property $LD_k$ on $p\Hs$, as does $M$. Arguing as in the proof of Lemma \ref{lem4.1} (i), we see that $M$ has property $LD_{k+1}$ on $\Hs$, and it follows that $\pi(A)$ has property $LD_{k+1}$ on $\Hs$.
\end{proof}

\begin{lem}\label{lem2.2a}
Let $M$ be a properly infinite von Neumann
algebra, let $A$ be a ${\mathrm{C}}^*$-algebra and let $\Hs$ be a separable infinite dimensional Hilbert space.
\begin{itemize}
\item[{\rm{(i)}}]  $M$ has property $D_{3/2}$. 
\item[{\rm{(ii)}}] $A \otimes {\mathbb{K}}(\Hs)$ has property $D_{3/2}$.
\end{itemize}
\end{lem}
\begin{proof} (i) In our notation, \cite[Theorem 2.4]{Ch} shows that $M$ has
property $D_{3/2}^*$. To show that it has property $D_{3/2}$, consider a
nondegenerate representation $\pi$ of $M$ on some Hilbert space $\mathcal K$.
Then the weak closure  ${\overline{\pi(M)}}^{\,\mathrm{wot}}$ must be properly
infinite, otherwise $M$ would have a nonzero tracial state. Thus
${\overline{\pi(M)}}^{\,\mathrm{wot}}$ has property $LD_{3/2}$ on $\mathcal K$.
Since $\pi$ was arbitrary, we conclude that $M$ has property $D_{3/2}$.

\noindent  (ii) Consider an arbitrary nondegenerate representation $\pi$ of 
$A \otimes {\mathbb{K}}(\Hs)$. Its weak closure is properly infinite otherwise there would be a nonzero tracial state on this tensor product. The result now follows from (i).
\end{proof}

\section{Liminary C$^*$-algebras}\label{sec3}

Recall that a C$^*$-algebra $A$ is \emph{liminary} if every irreducible representation $\pi:A\rightarrow\mathbb B(\Hs)$ has $\pi(A)=\mathbb K(\Hs)$, the algebra of compact operators on $\Hs$.  These form the building blocks of type $\mathrm{I}$ C$^*$-algebras as discussed in the next section.  Our objective  in this section is to establish that C$^*$-algebras that are nearly contained in separable liminary C$^*$-algebras must themselves be liminary. We start by recording an easy observation for later use.

\begin{lem}\label{lem2.2}
\begin{itemize}
\item[{\rm (i)}] Let $J$ be an ideal in a ${\mathrm{C}}^*$-algebra $B$ and
let $A$ be a $\mathrm{C}^*$-subalgebra of $B$. If $A\subset_1 J$, then
$A\subseteq J$.
\item[\rm (ii)] If $A$ is a $\mathrm{C}^*$-subalgebra of
$\mathbb{B}(\mathcal{H})$ and $A\subset_1\mathbb{K}(\mathcal{H})$, then
$A\subseteq \mathbb{K}(\mathcal{H})$.
\end{itemize}
\end{lem}

\begin{proof}
Let $\pi:B\to B/J$ be the quotient map. By hypothesis, the restriction of
$\pi$ to $A$ is strictly contractive and so $\pi(A)=0$. Thus $A\subseteq J$.
This proves (i), and the second part is a special case that we have stated
separately for future reference.
\end{proof}

We will need to decompose general representations in terms of irreducible ones, leading naturally to direct integral theory for which a standard reference is \cite[Chapter~14]{KR2}. We briefly review the facts that we will need, and all are taken from the first two sections of \cite[Chapter~14]{KR2}.

A direct integral decomposition of a separable Hilbert space $\Hs$ consists of a measure space $(X,\mu)$ and Hilbert spaces $\Hs_t$, $t\in X$, so that $\Hs = \int^\oplus_X \Hs_t \,d\mu(t)$. Operators in ${\mathbb{B}}(\Hs)$ are diagonalizable if they act as scalar multiples of the identity on each $\Hs_t$, and are decomposable if they can be written $T = \int^\oplus_X T_t\, d\mu(t)$ for operators $T_t\in {\mathbb{B}}(\Hs_t)$ with suitable measurability requirements. For such a decomposable $T$, we have that $\|T\|$ is the essential supremum of the measurable function $t\to\|T_t\|$. The diagonalizable operators form a von~Neumann algebra which  can be identified with $L^\infty(X,\mu)$, and whose commutant is the von~Neumann algebra of decomposable operators. Given an abelian von~Neumann algebra $C$ on a separable Hilbert space $\Hs$, there is a direct integral decomposition of $\Hs$ for which $C$ becomes the algebra of diagonalizable operators.

Given a direct integral decomposition $\Hs = \int^\oplus_X \Hs_t\, d\mu(t)$ and a separable C$^*$-algebra $A$ contained in the algebra of decomposable operators, we fix a countable dense set $\{a_n\}^\infty_{n=1}$ in $A$ and write each $a_n$ as $\int^\oplus_X a_n(t) \, d\mu(t)$. Then we define $A(t)\subseteq {\mathbb{B}}(\Hs_t)$ to be the C$^*$-algebra generated by $\{a_n(t)\}^\infty_{n=1}$ for each $t\in X$. Any element $a\in A$ has a decomposition $\int^\oplus_X a_t\, d\mu(t)$ where $a_t\in A(t)$ for almost all $t\in X$. More generally, for a separable C$^*$-algebra $A$, each representation $\pi\colon \ A\to {\mathbb{B}}(\Hs)$ with range in the algebra of decomposable operators has a decomposition $\pi_t\colon \ A\to {\mathbb{B}}(\Hs_t)$ such that $\pi = \int^\oplus_X \pi_t\, d\mu(t)$, and $\pi(A)(t) = \pi_t(A)$ for almost all $t\in X$. Moreover, $\pi(A)' = \int^\oplus_X \pi_t(A)'\, d\mu(t)$,
and $\pi(A)'$ is the algebra of diagonalizable operators if and only if almost all $\pi_t$'s are irreducible. The following lemma establishes the link between near inclusions and direct integrals.

\begin{lem}\label{lem2.1}
Let $\Hs$ be a separable Hilbert space with direct integral decomposition
$\int^\oplus_X H_t\, d\mu(t)$. Let $A$ and $E$ be respectively a separable
${\mathrm{C}}^*$-subalgebra and a separable operator subspace of the algebra of
decomposable operators. If $A\subseteq_\delta E$ for some constant $\delta$,
then $A(t) \subseteq_\delta E(t)$ for almost all $t\in X$.
\end{lem}

\begin{proof}
Let $D$ be the separable C$^*$-algebra generated by $A$ and $E$, and fix a
countable dense set $\{d_n\}^\infty_{n=1}$ in $D$ which includes countable
dense subsets of $A$ and $E$. By taking their span over the rational
field, we may assume that this is a listing of the elements in a countable
$\mathbb{Q}$-subspace. Choose representations $d_n(t)$, $t\in X$ for each $d_n$.
By removing a countable number of null sets, we may assume that
$\|d_n\|=\sup\{\|d_n(t)\|\colon t\in X\}$. This ensures that any Cauchy
sequence from this set is pointwise Cauchy, allowing us to choose
representations $d(t)$ for each $d\in D$ to satisfy
$\|d\|=\sup\{\|d(t)\|\colon t\in X\}$. Removal of another countable collection
of null sets allows us to  assume that $d\mapsto d(t)$ defines a *-homomorphism
for $d\in D$.

Now let $A(t)$ be the C$^*$-algebra generated by $\{a(t)\colon a\in A\}$ for
$t\in X$, while $E(t)$ is the operator space generated by $\{e(t)\colon e\in
E\}$. By proximinality of ideals in C$^*$-algebras \cite[Prop. II.1.1]{HWW},
given $t\in X$ and $y\in
A(t)$, $\|y\|=1$, there exists $a\in A$, $\|a\|=1$, such that $a(t)=y$. By
hypothesis we can choose $e\in E$ so that $\|a-e\|\leq \delta$, and so
$\|a(t)-e(t)\|\leq \delta$. This shows that $A(t)\subseteq_{\delta}E(t)$ for
all $t\in X$ with the possible exception of the countable number of deleted
null sets.
\end{proof}

We can now prove the main theorem in this section.

\begin{thm}\label{thm3.1}
Let $B$ be a separable liminary ${\mathrm{C}}^*$-algebra and let $A$ be a ${\mathrm{C}}^*$-algebra such that $A \subseteq_\delta B$ for some $\delta<1/201$. Then $A$ is liminary.
\end{thm}

\begin{proof}
We first observe that the hypotheses imply that $A$ is separable (see \cite[Proposition 2.10]{CSSWW}, for example).  We fix an arbitrary irreducible representation $\pi\colon\ A\to {\mathbb{B}}(\Hs)$, and we must now show that $\pi(A) \subseteq {\mathbb{K}}(\Hs)$. Since irreducible representations arise from the GNS representations of pure states, there is a pure state $\phi$ on $A$ and a unit vector $\xi\in \Hs$ so that
\begin{equation}\label{eq3.1}
 \phi(a) = \langle\pi(a)\xi,\xi\rangle,\qquad a\in A.
\end{equation}
The GNS representation of any Hahn--Banach extension of $\phi$ to a state on $C^*(A,B)$ gives a representation $\sigma\colon \ C^*(A,B)\to {\mathbb{B}}(\Ks)$ where $\Hs\subseteq \Ks$ and $\pi(a)$ is the restriction to $\Hs$ of $\sigma(a)$ for each $a\in A$. Since $B$ is separable, so also is $\Ks$. 

Let $q$ denote the orthogonal projection from $\Ks$ onto $\Hs$, which lies in
$\sigma(A)'$.  Let $e$ be the central support  of $q$ in $\sigma(A)'$, which
also lies in $\overline{\sigma(A)}^{\,\mathrm{wot}}$. Now
$\overline{\sigma(A)}^{\,\mathrm{wot}}e\cong
\overline{\sigma(A)}^{\,\mathrm{wot}}q=\mathbb B(\Hs)$ and hence
$\overline{\sigma(A)}^{\,\mathrm{wot}}e$ is a type I von Neumann algebra and so
is injective. Since $\sigma(A)\subseteq_\delta\sigma(B)$, a near inclusion
version of the Kaplansky density argument of \cite[Lemma 5]{KK} (which works in
exactly the same way as the original argument) gives 
\begin{equation}
\overline{\sigma(A)}^{\,\mathrm{wot}}e\subseteq
\overline{\sigma(A)}^{\,\mathrm{wot}}\subseteq_\delta\overline{\sigma(B)}^{\,
\mathrm {wot}}
\end{equation}
Adjoining units as in \cite[Theorem 6.1]{Ch2}, it follows that 
\begin{equation}
W^*(\overline{\sigma(A)}^{\,\mathrm{wot}}e,I_\Ks)\subseteq_{2\delta}
W^*(\overline { \sigma(B)}^{\,\mathrm{wot}},I_\Ks)=\sigma(B)''.
\end{equation}
Write $R=W^*(\overline{\sigma(A)}^{\,\mathrm{wot}}e,I_\Ks)$.  Since
$2\delta<1/100$
by hypothesis, Theorem 4.3 of \cite{Ch2} (which has an implicit shared unit
hypothesis) applies to the near inclusion $R\subseteq_{2\delta}\sigma(B)''$.
Thus  there is a unitary $u\in (R\cup \sigma(B))''$ such that $uRu^* \subseteq
\sigma(B)''$ and 
\begin{equation}
\|uru^*-r\|\leq200\,\delta\,\|r\|,\qquad r\in R.
\end{equation}
Let $\tilde\sigma$ be the representation of $B$ given by
$\tilde\sigma(b)=u^*\sigma(b)u$ so that $R\subseteq \tilde\sigma(B)''$. We now
show that $\sigma(A)\subseteq_\gamma\tilde\sigma(B)$, where
$\gamma=201\delta<1$, by the choice of the bound on $\delta$. Indeed, for $a\in
A$, choose $b\in B$ with $\|a-b\|\leq \delta\|a\|$, giving the inequalities 
\begin{align}
\|\sigma(a)-\tilde{\sigma}(b)\|&\leq\|\sigma(a)-u^*\sigma(a)u\|+\|u^*\sigma(a)u-u^*\sigma(b)u\|\nonumber\\
&\leq 200\delta\|\sigma(a)\|+ \|\sigma(a)-\sigma(b)\|\leq 201\delta\|\sigma(a)\|.
\end{align}
Since $B$ is liminary, both $\tilde\sigma(B)''$ and $\tilde\sigma(B)'$ are type
I. Let $f$ be an abelian projection in $\tilde\sigma(B)'$ with central support
$I$ in this algebra. As $\tilde\sigma(B)'\subseteq R'$, the central support of
$f$ in $R'$ is also $I$, and hence $ef\neq 0$. Consider the representation
$\beta:B\rightarrow\mathbb B(f(\Ks))$ given by
$\beta(b)=\tilde\sigma(b)|_{f(\Ks)}=f\tilde\sigma(b)$ for $b\in B$.  Identifying
the abelian von Neumann algebra $\beta(B)'$ with $L^\infty(X,\mu)$ for some
measure space $(X,\mu)$, we can decompose $f(\Ks)$ as the direct integral 
\begin{equation}
f(\Ks)=\int_X^\oplus \Ks_t\,d\mu(t).
\end{equation}
In this way $\beta(B)''$ is the algebra of decomposable operators and $\beta$ decomposes as a direct integral
\begin{equation}
\beta=\int_X^\oplus\beta_t\,d\mu(t),
\end{equation}
where, after deleting a null set we may assume that each $\beta_t:B\rightarrow\mathbb B(\Ks_t)$ is irreducible so that $\beta_t(B)=\mathbb K(\Ks_t)$.

As $R\subseteq \tilde\sigma(B)''$, the representation $\alpha:A\rightarrow\mathbb B(\Ks)$ given by $\alpha(a)=\sigma(a)ef$ maps $A$ into the decomposable operators. Thus we may disintegrate $\alpha$ as 
\begin{equation}
\alpha=\int_X^\oplus \alpha_t\,d\mu(t).
\end{equation}
Similarly we decompose $ef$ as $\int_Xe_t\,d\mu(t)$. Since
$\sigma(A)\subseteq_\gamma\tilde\sigma(B)$, we have
$\alpha(A)=\sigma(A)ef\subseteq_\gamma ef\tilde\sigma(B)ef=e\beta(B)e$, noting
that  $e\beta(B)e$ is  an
operator space but not necessarily  a C$^*$-algebra because  $e$ need not
commute with $\tilde\sigma(B)$.  Then for almost all $t$, Lemma \ref{lem2.1}
(which is formulated to allow near inclusions with operator spaces) gives
\begin{equation}
\alpha_t(A)=\alpha_t(A)e_t\subseteq_\gamma e_t\mathbb K(\Ks_t)e_t=\mathbb K(e_t(\Ks_t)).
\end{equation}
As $\pi$ is irreducible, the projection $q$ is minimal in $\sigma(A)'e$ and has
central support $e$ by definition. Since $ef$ is also in $\sigma(A)'e$ with
central support $e$, comparison theory gives a partial isometry $v\in
\sigma(A)'e$ with $v^*v=q$ and $vv^*\leq fe$.  Let $\eta=v\xi$, so that
$\|\eta\|=1$ and
\begin{equation}
\phi(a)= \langle\sigma(a)q\xi,q\xi\rangle=\langle \sigma(a)ev^*\eta,v^*\eta\rangle=\langle \alpha(a)\eta,\eta\rangle,\qquad a\in A.
\end{equation}
Writing the  decomposition of $\eta$ as  $\eta=\int_X\eta_t\,d\mu(t)$, it
follows that 
\begin{equation}\label{neweq1}
\phi(a)=\int_X\langle \alpha_t(a)\eta_t,\eta_t\rangle \,d\mu(t),\qquad a\in A.
\end{equation}

Recall that a left ideal $J$ in a C$^*$-algebra $A$ is \emph{regular} if there exists $e\in A$ with $xe-x\in J$ for all $x\in A$. Consider the left ideal $J=\{a\in A:\phi(a^*a)=0\}$. Since $\phi$ is a pure
state on $A$, $J$ is a maximal regular left ideal of $A$ (see \cite[Proposition
3.13.6]{Ped}). Since this ideal is separable, we may remove a set of measure
zero to ensure that $\langle \alpha_t(a^*a)\eta_t,\eta_t\rangle=0$ for all $a\in
J$ and all $t$. It then follows that, for each $t$, the positive linear
functional $\psi_t(a)=\langle \alpha_t(a)\eta_t,\eta_t\rangle$ satisfies
$\psi_t(a^*a)=0$ for all $a\in J$.  Fix $t_0$ such that $\psi_{t_0}\neq 0$; such
a $t_0$ must exist by (\ref{neweq1}). By maximality of $J$, the
ideal $\{a\in A:\psi_{t_0}(a^*a)=0\}$ must be $J$ so that, by \cite[Proposition
3.13.6]{Ped}, (which sets out the non-unital version of  \cite[Theorem
2]{KadTrans}) $\psi_{t_0}$ is a scalar multiple of $\phi$.  Thus
there is a constant $c>0$ so that
\begin{equation}
\phi(a)=c\langle \alpha_{t_0}(a)\eta_{t_0},\eta_{t_0}\rangle,\qquad a\in A.
\end{equation}
Consider the projection $p_{t_0}$ onto $\overline{\alpha_{t_0}(A)\eta_{t_0}}$ in
$\Ks_{t_0}$.  By uniqueness of the GNS construction, $\alpha_{t_0}(\cdot)p_0$ is
equivalent to $\pi$. Since $\alpha_{t_0}(A)\subseteq\mathbb K(\Ks_{t_0})$,
it follows that $\pi(A)\subseteq \mathbb K(\Hs)$, completing the proof that $A$
is liminary.
\end{proof}

\section{Type $\mathrm{I}$ C$^*$-algebras}\label{sec4}

 In this section we consider near inclusions of C$^*$-algebras $A\subseteq_\gamma B$ where $B$ is type I. The main objective is to prove that $A$ is also of type I under suitable hypotheses. It will then follow from 
\cite[Corollary 4.4]{CSSWW} that $A$ embeds into $B$.

Recall that a positive element $x$ in a C$^*$-algebra $B$ is abelian if the norm closure of the algebra $xBx$ is commutative. We then say that $B$ is type $\mathrm{I}$ if each nonzero quotient of $B$ contains a nonzero abelian element.  A composition series for a C$^*$-algebra $B$ is a strictly increasing set of ideals $I_\alpha$ indexed by a segment $\{0\le \alpha\le\beta\}$ of the ordinals such that $I_0=0$, $I_\beta=B$, and for each limit ordinal $\gamma$, $I_\gamma$ is the norm closure of $\bigcup_{\alpha<\gamma} I_\alpha$. The connection between type~I C$^*$-algebras and liminary C$^*$-algebras is exhibited in \cite[Theorem~6.2.6]{Ped} (see also \cite{Fell}) where it is shown that $B$ is type $\mathrm{I}$ if and only if it has a composition series such that $I_{\alpha+1}/I_\alpha$ is liminary for each $\alpha<\beta$. There are many other equivalent formulations in the separable case, \cite[Theorem~6.8.7]{Ped}, many of which hold generally.
 
To make use of the composition series, we will need to consider the ideal structures of $A$ and $B$ when $A \subseteq_\gamma B$. In the related context of close C$^*$-algebras with $d(A,B) \le \gamma$, this has already been done by Phillips, \cite{Ph},  who showed that the lattices of ideals in $A$ and $B$ are isomorphic. For genuine containments $A\subseteq B$, every ideal $I$ in $B$ induces an ideal $I\cap A$ in $A$. Given a near containment $A\subseteq_\gamma B$ and an ideal $I$ in $B$, we aim to produce a corresponding ideal $J$ in $A$ which is nearly contained in $I$ so that we can represent $A/J$ and $B/I$ on the same space with $A/J$ nearly contained in $B/I$.  Our methods rely heavily on those of Phillips, but with some modifications. In particular, we will require a near containment of $B'$ in $A'$ to produce these ideals.  This was not needed in \cite{Ph} as close C$^*$-algebras have close centres. 

\begin{lem}\label{lem4.3}
Let $\gamma$ and $k$ be positive constants. Suppose that $A$ and $B$ are ${\mathrm{C}}^*$-algebras such that $A \subseteq_\gamma B$ and $A$ has property $D_k$. Given an ideal $I$ in $B$, there exists an ideal $J$ in $A$ with the following properties:
\begin{itemize}
 \item[\rm (i)] $J \subseteq_{\gamma'}I$, where $\gamma' = (4k+6)\gamma$.
\item[\rm (ii)] There exist faithful representations $\rho$ of $A/J$ and $\sigma$ of $B/I$ on the same Hilbert space $\Ks$ such that
$ \rho(A/J) \subseteq_{\gamma''} \sigma(B/I)$,
where $\gamma'' = (4k+5)\gamma$.
\item[\rm (iii)] Let $\lambda >0$ be a fixed but arbitrary constant. If $x\in A$ satisfies $d(x,I) \le \lambda\gamma\|x\|$, then $d(x,J) \le (4k+4+\lambda)\gamma\|x\|$.
\end{itemize}

\noindent Under an additional assumption, the following also holds:
\begin{itemize}
\item[\rm (iv)]
Suppose now that $(8k+10)\gamma<1$.  Let $I_1\subseteq I_2 \subseteq B$ be ideals and let $J_1$ and $J_2$ be   corresponding choices of ideals in $A$ satisfying {\rm{(i)--(iii)}}. Then $J_1 \subseteq J_2$, and these choices of ideals in $A$ are unique.
\end{itemize}
\end{lem}

\begin{proof}
Take a representation $\pi$ of $B$ with kernel $I$ and extend to a representation $\tilde\pi$ of $C^*(A,B)$ on some larger Hilbert space $\Ks$, so that $\tilde\pi|_B$ contains $\pi$ as a summand. Thus there is a projection $p\in \tilde\pi(B)'$ so that $\pi  = p\tilde\pi|_Bp$. Since $A$ has property $D_k$, $\tilde\pi(A)$ has property $LD_{k+1}$ from Lemma \ref{lem2.3b}, so $\tilde\pi(B)' \subseteq_{(2k+2)\gamma} \tilde\pi(A)'$ from \cite[Proposition 2.5]{CSSW}. Thus there is a projection $q\in\tilde\pi(A)'$ such that $\|p-q\|\le (4k+4)\gamma$ from \cite[Proposition 3.1]{CSSW}. Let $J$ be the kernel of the representation $x\mapsto \tilde\pi(x)q$ for $x\in A$.

We now show that $J$ is nearly contained in $I$, and to this end consider $x\in J$, $\|x\|=1$. We may choose $y\in B$ such that $\|y-x\|\le \gamma$, since $J\subseteq_\gamma B$. Then
\begin{align}
 \|p\tilde \pi(y)\| &= \|p\tilde\pi(y) - q\tilde\pi(x)\|
\le \|p\tilde \pi(y) - p\tilde\pi(x)\| + \|p-q\|\notag\\
\label{eq4.2a}
&\le \gamma+(4k+4)\gamma.
\end{align}
Since ideals are proximinal \cite[Prop. II.1.1]{HWW}, we may choose $z\in I$ so that $\|p\tilde\pi(y)\| = \|y-z\|$. Then
\begin{equation}\label{eq4.3}
 \|x-z\| \le \|x-y\| + \|y-z\| \le 2\gamma + (4k+4)\gamma,
\end{equation}
showing that $J \subseteq_{(4k+6)\gamma}I$. This establishes (i).

The representations $a\mapsto \tilde\pi(a)q$, $a\in A$, and $b\mapsto\tilde\pi(b)p$, $b\in B$, induce faithful representations $\rho$ of $A/J$ and $\sigma$ of $B/I$ on $\Ks$. We now show that $\rho(A/J)$ is nearly contained in $\sigma(B/I)$. By proximinality of ideals we need only consider an element $x\in  A$ such that $\|x\| = \|x\|_{A/J}=1$. Then choose $y\in B$ so that $\|x-y\|\le \gamma$. It follows that
\begin{align}
 \|\tilde\pi(x)q-\tilde\pi(y)p\| &\le \|\tilde\pi(x-y)p +\tilde\pi(x)(q-p)\|\notag\\
\label{eq4.3a}
&\le \gamma+(4k+4)\gamma,
\end{align}
showing that $\rho(A) \subseteq_{(4k+5)\gamma}\sigma(B)$. This proves (ii).

Now fix $\lambda>0$ and consider $x\in A$ such that $d(x,I) \le \lambda\gamma\|x\|$. Given $\vp>0$, there exists $i\in I$ such that $\|x-i\|< (\lambda+\vp)\gamma\|x\|$. It follows that
\begin{equation}\label{eq4.4}
 \|\tilde\pi(x)p-\tilde\pi(i)p\| < (\lambda+\vp)\gamma\|x\|.
\end{equation}
Since $\tilde\pi(i)p=0$, we obtain $\|\tilde\pi(x)p\| < (\lambda+\vp)\gamma\|x\|$, so
\begin{align}
 \|\tilde\pi(x)q\| &\le \|\tilde\pi(x)(p-q)\| + \|\tilde\pi(x)p\|
< (4k+4)\gamma\|x\| + (\lambda+\vp)\gamma\|x\|\notag\\
\label{eq4.5}
&= (4k+4+\lambda+\vp)\gamma\|x\|.
\end{align}
Now $\|\tilde\pi(x)q\| = d(x,J)$, and (iii) follows from \eqref{eq4.5} since $\vp>0$ was arbitrary.

For the fourth part, we now make the additional assumption that $(8k+10)\gamma<1$.
Consider ideals $I_1 \subseteq I_2 \subseteq B$, and let $J_1$ and $J_2$ be ideals in $A$ so that the pairs $(I_1,J_1)$ and $(I_2,J_2)$ satisfy (i)--(iii). Let $j_1\in J_1$, $\|j_1\| =1$. From (i) we may choose $i_1\in I_1$ with $\|i_1-j_1\|\le (4k+6)\gamma = (4k+6)\gamma\|j_1\|$. Since $i_1\in I_2$, this gives $d(j_1,I_2) \le (4k+6)\gamma\|j_1\|$, and it follows from (iii) with $\lambda=4k+6$ that $d(j_1,J_2) \le (8k+10)\gamma<1$. Thus $J_1\subseteq_{(8k+10)\gamma} J_2$, and the containment $J_1\subseteq J_2$ follows from Lemma \ref{lem2.2} (i).

Finally suppose that $I$ is an ideal in $B$ so that there are two ideals $J_1$ and $J_2$ in $A$ satisfying (i)--(iii). The last argument, with $I_1=I_2=I$, gives $J_1\subseteq J_2$ and $J_2\subseteq J_1$, proving uniqueness of the choice of ideal in $A$. This establishes (iv).
\end{proof}

The following lemma will be used to transfer composition series between the algebras of a near inclusion.

\begin{lem}\label{lem4.4}
Let $\vp<1$ be a fixed positive constant. Let $\gamma$ and $k$ be positive constants satisfying
\begin{equation}\label{eq4.6}
 (16k^2+52k+40)\gamma \leq \vp.
\end{equation}
Let $A$ and $B$ be ${\mathrm{C}}^*$-algebras with $A \subseteq_\gamma B$, and suppose that $A$ has property $D_k$. Let $I_1 \subseteq I_2 \subseteq B$ be ideals and let $J_1 \subseteq J_2\subseteq A$ be the ideals constructed in Lemma~\ref{lem4.3} (iv), noting that $(8k+10)\gamma<1$. Then there is a Hilbert space $\Ks$ and faithful representations $\sigma\colon \ I_2/I_1\to {\mathbb{B}}(\Ks)$ and $\rho\colon  \ J_2/J_1\to {\mathbb{B}}(\Ks)$ such that
\begin{equation}
 \rho(J_2/J_1) \subseteq_\vp \sigma(I_2/I_1).
\end{equation}
\end{lem}

\begin{proof}
From Lemma~\ref{lem4.3} (i), $J_2 \subseteq_{\gamma'}I_2$ where $\gamma' = (4k+6)\gamma$. Since $J_2$ is an ideal in $A$, Lemma~\ref{lem4.1}~(ii) shows that $J_2$ has property $D_{k}$. Thus we may apply Lemma~\ref{lem4.3} to $J_2 \subseteq_{\gamma'} I_2$ and the ideal $I_1 \subseteq I_2$ with $\gamma$ replaced by $\gamma'$  to obtain an ideal $\widetilde J_1 \subseteq J_2 \subseteq A$ satisfying conditions (i)--(iii) with the following changes of constants:
\begin{itemize}
 \item[(I)] $\widetilde J_1 \subseteq_\delta I_1$ where $\delta=  (4k+6)\gamma'$.
\item[(II)] There exist faithful representations $\rho$ of $J_2/\widetilde J_1$ and $\sigma$ of $I_2/I_1$ on the same Hilbert space $\Ks$ such that 
 $\rho(J_2/\widetilde J_1) \subseteq_{\delta'} \sigma(I_2/I_1)$
where $\delta' = (4k+5)\gamma'$.
\item[(III)] Let $\lambda>0$ be a fixed but arbitrary constant. If $x\in J_2$ satisfies $d(x,I_1) \le \lambda\gamma'\|x\|$, then
\begin{equation}\label{eq4.7}
 d(x,\widetilde J_1) \le (4k+4+\lambda)\gamma'\|x\|.
\end{equation}
\end{itemize}

The result will follow from (II) provided that we can show that $\widetilde J_1=J_1$. To this end, consider $x\in J_1 \subseteq J_2$ with $\|x\|=1$. From Lemma~\ref{lem4.3} (i), 
\begin{equation}\label{eq4.8}
 d(x,I_1) \le (4k+6)\gamma = (4k+6)\gamma\|x\|.
\end{equation}
Taking $\lambda = (4k+6)\gamma/\gamma'$ in (III), we obtain
\begin{align}
 d(x,\widetilde J_1) &\le (4k+4 + (4k+6)\gamma/\gamma')\gamma'\notag\\
&=  (4k+6)(4k+5)\gamma\notag\\
\label{eq4.9}
&= (16k^2 + 44k +30)\gamma<\vp<1.
\end{align}
Lemma \ref{lem2.2} (i) then gives $J_1\subseteq \widetilde J_1$.

Now consider $x\in\widetilde J_1$, $\|x\|=1$. From (I), 
\begin{equation}\label{eq4.10}
 d(x,I_1) \le (4k+6)\gamma' = (4k+6)^2\gamma\|x\|.
\end{equation}
From Lemma \ref{lem4.3} (iii) with $\lambda =  (4k+6)^2$, we see that
\begin{align}
 d(x,J_1) &\le (4k+4+(4k+6)^2)\gamma\notag\\
\label{eq4.11}
&= (16k^2 + 52k +40)\gamma\leq \vp < 1.
\end{align}
Lemma \ref{lem2.2} (i) then gives the reverse containment  $\widetilde J_1\subseteq  J_1$. This proves that $J_1=\widetilde J_1$, and it now follows from (II) that there are faithful representations $\rho$ of $J_2/J_1$ and $\sigma$ of $I_2/I_1$ on a Hilbert space $\Ks$ such that
\begin{equation}
 \rho(J_1/J_1) \subseteq_{(16k^2+44k+30)\gamma} \sigma(I_2/I_1).
\end{equation}
Since $(16k^2+44k+30)\gamma<(16k^2+52k+40)\gamma\leq \vp$, the result follows.
\end{proof}

We now come to the main result of the paper. The inequality in the hypotheses is to ensure that applications of Theorem~\ref{thm3.1}, Lemma~\ref{lem4.3}, and Lemma~\ref{lem4.4} are valid.

\begin{thm}\label{thm4.5}
Let $k$ and $\gamma$ be positive constants satisfying
\begin{equation}
 (16k^2+52k+40)\gamma <  1/201.
\end{equation}
Let $B$ be a separable type {\rm I} ${\mathrm{C}}^*$-algebra and let $A$ be a ${\mathrm{C}}^*$-algebra such that $A$ has property $D_k$ and $A\subseteq_\gamma B$. Then $A$ is a type {\rm I} ${\mathrm{C}}^*$-algebra.
\end{thm}

\begin{proof}
From the discussion of type I C$^*$-algebras at the beginning of this section,
$B$ has a composition series $I_\alpha$, $0\le \alpha\le\beta$, where $I_0=0$,
$I_\beta=B$, and for each limit ordinal $\alpha$, $I_\alpha$ is the norm closure
of $\bigcup_{\alpha'<\alpha} I_{\alpha'}$. Moreover, each quotient
$I_{\alpha+1}/I_\alpha$ is liminary. Since $(8k+10)\gamma<1$, the ideals
$J_\alpha \subseteq A$ constructed from $I_\alpha\subseteq B$ in
Lemma~\ref{lem4.3} are nested, and $J_\beta=A$ from the uniqueness of each
$J_\alpha$. Applying Lemma~\ref{lem4.4} with $\vp =16k^2+52k+40< 1/201<1$, there
exist faithful representations $\rho_\alpha$ of $J_{\alpha+1}/J_\alpha$ and
$\sigma_\alpha$ of $I_{\alpha+1}/I_\alpha$ on the same Hilbert space so that
\begin{equation}
 \rho_\alpha(J_{\alpha+1}/J_\alpha) \subseteq_\vp \sigma_\alpha (I_{\alpha+1}/I_\alpha).
\end{equation}
Since $I_{\alpha+1}/I_\alpha$ is separable and liminary, it follows from Theorem~\ref{thm3.1} that each quotient $J_{\alpha+1}/J_\alpha$ is also liminary.

For a fixed limit ordinal $\alpha$, let $J = \ovl{\bigcup_{\alpha'<\alpha}J_{\alpha'}}$. Then $J \subseteq J_\alpha$, and we must establish the reverse inclusion. Consider $x\in J_\alpha$ with $\|x\|=1$. Then there exists $y\in I_\alpha$ such that $\|x-y\| \le (4k+6)\gamma$ from Lemma \ref{lem4.3} (i). At the cost of an error of say $\gamma$, we may assume that $y$ lies in some $I_{\alpha'}$ for $\alpha'<\alpha$, so that $\|x-y\| \le (4k+7)\gamma$. Applying Lemma~\ref{lem4.3} (iii) with $\lambda=(4k+7)$, we obtain $d(x,J)\le d(x,J_{\alpha'}) \le (8k+11)\gamma <  10^{-3}$. The containment $J_\alpha\subseteq J$ follows from Lemma \ref{lem2.2} (i), proving equality. After deleting any duplicates from the list, $\{J_\alpha\colon \ 0\le\alpha \le \beta\}$ is a composition series for $A$ with liminary quotients.
We have now proved that $A$ is type I.
\end{proof}

\begin{cor}\label{cor4.5}
Let $k$ and $\gamma$ be positive constants satisfying
\begin{equation}
 (16k^2+52k+40)\gamma <  1/201,\ \ \gamma<1/420000.
\end{equation}
Let $B$ be a separable type {\rm I} ${\mathrm{C}}^*$-algebra and let $A$ be a ${\mathrm{C}}^*$-algebra such that $A$ has property $D_k$ and $A\subseteq_\gamma B$. Then $A$
embeds into $B$.
\end{cor}
\begin{proof}
From Theorem \ref{thm4.5}, $A$ is a type I C$^*$-algebra and so it is nuclear. The result now follows from \cite[Corollary 4.4]{CSSWW}.
\end{proof}

The similarity problem asks whether every representation of a C$^*$-algebra is similar to a $*$-representation. This is an open question which is equivalent to the existence of a constant $k_0$ so that every C$^*$-algebra has property $D_{k_0}$. This would allow us to remove the property $D_k$ hypothesis from Theorem~\ref{thm4.5}. We now discuss another situation where this is possible.

It is also natural to consider closeness and near inclusions in the completely bounded sense (see the discussion at the end of Section 4 of \cite{CSSW}). Specifically, we say that $d_{{\mathrm{cb}}}(A,B) = d(A\otimes {\mathbb{K}}(\Hs), B\otimes {\mathbb{K}}(\Hs))$, and define a cb-near inclusion  $A \subseteq_{{\mathrm{cb}},\gamma} B$ to mean $A\otimes {\mathbb{K}}(\Hs)\subseteq_\gamma B\otimes {\mathbb{K}}(\Hs)$ where $\Hs$ is a separable infinite dimensional Hilbert space. In this setting we have the following.

\begin{thm}\label{thm4.6}
Let $\gamma < 1/420000$, let $B$ be a separable type {\rm I} ${\mathrm{C}}^*$-algebra, and let $A$ be a ${\mathrm{C}}^*$-algebra such that $A \subseteq_{{\mathrm{cb}},\gamma} B$. Then $A$ is type {\rm I} and embeds into $B$.
\end{thm}

\begin{proof}
By definition, we have $A\otimes {\mathbb{K}}(\Hs) \subseteq_\gamma B\otimes {\mathbb{K}}(\Hs)$. By part (ii) of Lemma~\ref{lem2.2a}, $A \otimes {\mathbb{K}}(\Hs)$ has property $D_{3/2}$, so direct calculation shows that the hypotheses of Theorem~\ref{thm4.5} are satisfied with $k=3/2$. Thus $A\otimes {\mathbb{K}}(\Hs)$ is type I and so $A$ is nuclear. It follows from \cite[Corollary 4.4]{CSSWW} that $A$ embeds into $B$, whereupon we also see that it is type I.
\end{proof}

We conclude with some remarks on the paper \cite{J} by Johnson and the
connections to our results. There the author considers near containments
$A\subseteq_{\gamma} B$ where $B$ is a separable $n$-subhomogeneous
C$^*$-algebra (which is automatically type I). For suitably small $\gamma$,
$A$ is also $n$-subhomogeneous, so an embedding of $A$ into $B$ follows from
Corollary \ref{cor4.5} since nuclear C$^*$-algebras have property $D_1$.
Johnson obtains such an embedding, but also shows that it can be achieved
with conjugation by a unitary $u$ satisfying $\|u-I\|\leq f(\gamma, n)$,
where $f$ is a function satisfying $\lim_{\gamma\to 0}f(\gamma, n)=0$ for
each fixed value of $n$. Such a result cannot hold in the type I or even liminary situation due to the examples of $c_0\subseteq_{\gamma} C[0,1]\otimes \mathbb{K}(\Hs)$ presented in \cite{Johnson.PerturbationExample}. It would be interesting to characterise exactly which near inclusions can be spatially implemented by a unitary close to $I$.

\end{document}